\documentclass[11pt]{article}

\usepackage[utf8]{inputenc} 
\usepackage[english]{babel}
\usepackage{geometry}
\usepackage{amsmath, amsthm, amsfonts, amssymb, tikz-cd, mathtools, mathrsfs,theoremref, txfonts}
\usepackage{tgpagella} 
\usepackage[T1]{fontenc}
\usepackage{graphicx}
\usepackage{hyperref}

\newtheorem{theorem}{Theorem}[section]
\newtheorem{conjecture}[theorem]{Conjecture}

\newtheorem{lemma}[theorem]{Lemma}

\newtheorem{claim}[theorem]{Claim}
\theoremstyle{definition}

\newtheorem{remark}[theorem]{Remark}

\newtheorem{question}[theorem]{Question}

\usepackage{tcolorbox}
\newtcolorbox{mybox}{colback=red!5!white,colframe=red!75!black}
\newtcolorbox{alebox}{colback=blue!5!white,colframe=blue!75!black}

\newcommand{\R}{\mathbb{R}}

\newcommand{\ep}{\epsilon}

\DeclarePairedDelimiter\abs{\lvert}{\rvert}

\usepackage{titlesec}
\titleformat{\paragraph}
{\normalfont\normalsize\bfseries}{\theparagraph}{1em}{}
\titlespacing*{\paragraph}
{0pt}{3.25ex plus 1ex minus .2ex}{1.5ex plus .2ex}

\usepackage{lmodern}
\usepackage{xcolor}

\def\eps{{\varepsilon}}

\newcommand{\parag}[1]{\vspace{2mm}

\noindent{\bf #1} }

\newcommand{\ignore}[1]{}

\title{Progress on Local Properties Problems of Difference Sets\thanks{This research project was done as part of the 2020 NYC Discrete Math REU, supported by NSF awards DMS-1802059, DMS-1851420, DMS-1953141, and DMS-2028892.}}
\author{Anqi Li\thanks{Massachusetts Institute of Technology, 77 Massachusetts Ave, Cambridge, MA 02139.
{\sl anqili@mit.edu.} Supported by MIT Department of Science.}}

\begin{document}
\date{} 
\maketitle

\begin{abstract}
We derive several new bounds for the problem of difference sets with local properties, such as establishing the super-linear threshold of the problem. For our proofs, we develop several new tools, including a variant of higher moment energies and a Ramsey-theoretic approach for the problem. 
\end{abstract}

\section{Introduction}

Erd\H{o}s and Shelah \cite{E75} initiated the study of graphs with local properties. 
For parameters $n,k,$ and $\ell$, they considered edge colorings of the complete graph $K_n$,  such that every $K_k$ subgraph contains at least $\ell$ colors. 
They defined $f(n,k,\ell)$ as the minimum number of colors that such an edge coloring can have. 
For example, $f(n,3,2)$ is the minimum number of colors that is required to color $K_n$, so that it does not contain a monochromatic triangle.  

The value of $f(n,k,2)$ can be thought of as an inverse of Ramsey's theorem.
In Ramsey's theorem, we look for the largest $n$ such that there exists a coloring of the edges of $K_n$ with $c$ colors and no monochromatic $K_k$. 
When studying $f(n,k,2)$, we are given the number of vertices $n$, and instead look for the number of colors $c$.  

By definition, $1 \leq f(n,k,\ell) \leq \binom{n}{2}$. 
We usually consider $n$ to be asymptotically large, while $k$ and $\ell$ are fixed constants.
For asymptotically large $n$ and a fixed $k$, the \emph{linear threshold} is the smallest $\ell$ for which $f(n,k,\ell)=\Omega(n)$.
This value of $\ell$ is expressed as a function of $k$.
Similarly, the \emph{quadratic threshold} is the smallest $\ell$ that satisfies $f(n,k,\ell)=\Theta(n^2)$.
The \emph{polynomial threshold} is the smallest $\ell$ that satisfies $f(n,k,\ell)=\Omega(n^\eps)$ for some $\eps>0$.

Erd\H{o}s and Gy\'arf\'as \cite{EG97} identified the linear and quadratic thresholds to be $\ell=\binom{k}{2} -k + 3$ and $\ell=\binom{k}{2} - \lfloor k/2 \rfloor +2$, respectively. 
In the same paper, they also observed that $f(n,k,k)$ is polynomial in $n$. Conlon, Fox, Lee and Sudakov \cite{CFLS15} proved that $f(n,k,k-1)$ is sub-polynomial in $n$, which established the polynomial threshold $\ell=k$. Beyond these thresholds, almost tight bounds for $f(n,k,\ell)$ are known in other regimes. Particularly, Pohoata and Sheffer \cite{PS19} showed that, for any integers $k > m \geq 2$,
\begin{equation}\label{eq:PS19main}
    f\left(n,k, \binom{k}{2} - m \cdot \left\lfloor \frac{k}{m+1} \right\rfloor + m + 1 \right) = \Omega \left( n^{1 + \frac{1}{m}} \right).
\end{equation}
Probabilistic constructions of Erd\H{o}s and Gy\'arf\'as \cite{EG97} show that \eqref{eq:PS19main} is asymptotically tight up to sub-polynomial factors.

Zeev Dvir recently suggested an additive combinatorics variant of the local properties problem. 
As stated in \cite[Section 9]{EG97}, Erd\H{o}s and S\'os also studied this problem, but did not publish their results.
We define the \emph{difference set} of a set $A\subset \R$ as
\[ A-A = \{a-a'\ :\ a,a'\in A\ \text{ and }\ a-a'>0\}. \]
The standard definition of a difference set does not include the condition $a-a'>0$. 
However, this less common definition makes the following local properties problem nicer to study. 

For parameters $n,k,$ and $\ell$, we consider sets $A$ of $n$ real numbers, such that every subset $A'\subset A$ of size $k$ satisfies that $|A'-A'|\ge \ell$. 
Let $g(n,k,\ell)$ be the minimum size of $\abs{A-A}$, over all such sets $A$. 
In the current work, we study $g(n,k,\ell)$. The trivial bounds for this problem are $n-1\le g(n,k,\ell)\le \binom{n}{2}$.
As before, for a fixed $k$, the quadratic threshold is the smallest $\ell$ for which $g(n,k,\ell)=\Theta(n^2)$. 
We define the \emph{super-linear threshold} as the largest $\ell$ for which $g(n,k,\ell)=O(n)$.

\parag{Previous bounds.}
Currently, not much is known about the behavior of $g(n,k,\ell)$. The unpublished work of Erd\H{o}s and S\'os proved that $g(n,4,5)\geq \binom{n}{2} - n +2$. 
We also have the bound $g(n,k,\ell)\ge f(n,k,\ell)$, for every $n,k,$ and $\ell$. 
For example, the bounds of \eqref{eq:PS19main} also hold after replacing $f(\cdot)$ with $g(\cdot)$.
Indeed, we can construct a complete graph with a vertex for every element of $A$.
The color of an edge corresponds to the difference of the numbers that correspond to the two vertices. If every subset $A'\subset A$ of size $k$ satisfies that $|A'-A'|\ge \ell$, then every subgraph $K_k$ spans at least $\ell$ colors. 
Similarly, the total number of colors is $|A-A|$. 

Fish, Pohoata, and Sheffer \cite{FPS20} observed that Roth's theorem implies the following bound.
\begin{theorem}
For every $\ep > 0$, there exists $c$ that satisfies the following. For every sufficiently large integer $k$, we have that
\[ g\left(n,k, c \cdot k \cdot \log^{1/4 - \ep}k \right) = n 2^{O(\sqrt{\log n})}. \]
\end{theorem}

Fish, Lund, and Sheffer \cite{FLS19} obtained the following result by projecting a hypercube onto a line.

\begin{theorem} \label{thm:FLS19-g}
    For every $n > k \geq 0$ where $n$ is a power of two, we have that 
    \[  g\left(n,k, \frac{k^{\log_2(3)}-1}{2} \right) = O\left( n^{\log_2(3)}\right). \]
\end{theorem}

\parag{Our results.}
We derive several new bounds for $g(n,k,\ell)$. 
First, we identify the super-linear threshold of $g(n,k,\ell)$.
Recall that $\omega(\cdot)$ means ``asymptotically strictly larger than''.

\begin{theorem}\label{th:SuperLinear}
For every $k$ and sufficiently large $n$, we have that
\begin{align*} 
g(n,k,k-1) = n-1 \qquad \text{ and } \qquad g(n,k,k) =\omega(n).
\end{align*}
\end{theorem}

After establishing the super-linear threshold, it is natural to ask how quickly the bound changes as $\ell$ grows. We obtain  the following bound via a probabilistic argument. 

\begin{theorem} \label{th:ProbConst}
For every $c\ge 2$ and integer $k> (c^2+1)^2$, we have that
\[ g(n, k, ck+1) =O\Big(  n^{1 + \frac{c^2+1}{k}}\Big). \]
\end{theorem}
When fixing $c\ge2$, the bound of Theorem \ref{th:ProbConst} becomes arbitrarily close to $O(n)$ as $k$ grows.
We note that the assumption $k> (c^2+1)^2$ can be removed from the statement of this theorem. Indeed, when $k\le (c^2+1)^2$, the bound of the theorem is trivial.

Additive combinatorics provides many tools for studying sets that have a small difference set. 
We rely on such tools to study the case of small $\ell$, in the proofs of Theorems \ref{th:SuperLinear} and \ref{th:ProbConst}.
To study sets with larger values of $\ell$, we have to develop new tools. In particular, we obtain the following result by studying a new variant of higher moment energies. For the definition of higher moment energies and of our new variant, see Section \ref{sec:Dumbbell}.

\begin{theorem} \label{th:Lower43}
Every $k\ge 8$ that is a multiple of 8 satisfies  
\[ g\left(n, k, \frac{9k^2}{32} - \frac{9k}{8} + 5 \right) = \Omega(n^{4/3}). \]
\end{theorem}

By \eqref{eq:PS19main}, for the graphs local properties problem, we have that
\begin{equation*}
    f\left(n,k, \binom{k}{2} - 3 \cdot \left\lfloor \frac{k}{4} \right\rfloor + 4 \right) = \Omega \left( n^{4/3} \right).
\end{equation*}
This bound is tight up to subpolynomial factors. Theorem \ref{th:Lower43} shows that, in the additive combinatorics variant, a much weaker local restriction leads to $\Omega(n^{4/3})$, up to subpolynomial factors.  
Unlike the graph variant, in this case the coefficient of $k^2$ is smaller than $1/2$.

Theorem \ref{thm:FLS19-g}  provides an upper bound for the local properties problem by introducing a construction. Surprisingly, we show that the same construction can be used to obtain a proof for a lower bound. 

\begin{theorem} \label{th:CubeLower}
For every $k \geq 8$ that is a power of two and sufficiently large $n$, we have that
\[ g\left(n,k, \frac{k^{\log_2(3)}+1}{2} \right) = \Omega \left( n^{1+\frac{2}{k-2} }\right).\]
\end{theorem}

The proof of Theorem \ref{th:CubeLower} is in some sense Ramsey theoretic:  We show that a set with many repeated differences must contain a copy of the construction from Theorem \ref{thm:FLS19-g}. 
Such a subset contradicts the local property, and thus cannot exist.

Finally, we mention the best bound that we managed to obtain for the quadratic threshold.

\begin{claim} \label{cl:Quadratic}
When $k$ is a multiple of two, we have that 
\[ g\left(n, k, \frac{3k^2}{8} - \frac{3k}{4} + 2 \right) = \Omega(n^2). \]
\end{claim}

It is not difficult to prove Claim \ref{cl:Quadratic}, and it is possible that the quadratic threshold is much smaller. We include Claim \ref{cl:Quadratic} for completeness. That is, to document all the current best bounds for the arithmetic variant of the local properties problem. 

Lastly, we note that while most results in this paper apply to the setting of a general Abelian group, we decided to not pursue this direction and leave this to future work. 

\medskip

\noindent \textbf{Outline.} 
In Section \ref{sec:Dumbbell}, we define our new energy variant and use it to prove Theorem \ref{th:Lower43}. 
In Section \ref{sec:MoreLargeEll}, we prove our other results for large values of $\ell$: Theorem \ref{th:CubeLower} and Claim \ref{cl:Quadratic}.
In Section \ref{sec:SmallEll}, we prove our bounds for small values of $\ell$:
Theorem \ref{th:SuperLinear} and Theorem \ref{th:ProbConst}. In Section \ref{sec:Future}, we pose some questions for future work. 

\medskip \noindent
\textbf{Acknowledgements.} We would like to thank MIT Department of Science for the generous support. We would also like to express our gratitude towards Adam Sheffer for his mentorship and support throughout the NYC Discrete Math REU, as well as his invaluable help in the preparation of this paper. In addition, we would like to thank Sara Fish for many helpful discussions as well as Ilya Shkredov for his enlightening suggestion that led to an improvement in the bound of Theorem 1.5. Lastly, we thank the referees for their many useful comments and suggestions. 

\section{Dumbbell energy and Theorem \ref{th:Lower43}} \label{sec:Dumbbell}

In this section, we prove Theorem \ref{th:Lower43}. Our proof relies on a new energy variant. Before we introduce this energy, we briefly recall the notion of higher moment energies. 

The \emph{additive energy} of a finite set $A\subset \R$ is  
\[ E^+(A) = \abs{\{ (a_1, a_2, a_3, a_4) \in A^4\ :\ a_1 - a_2 = a_3 - a_4 \}}. \]
Additive energy is a common tool in additive combinatorics. 
For example, it provides the following lower bound for the size of the difference set:
\[ |A-A|\ge \frac{|A|^4}{E^+(A)}. \]

In the past decade, several variants of additive energy led to interesting new results. 
One of those variants is \emph{higher moment energies}. 
For an integer $\ell>1$, the $\ell$th moment energy of $A$ is   
\[ E_{\ell}^+(A) = \abs{\{ (a_1, \ldots, a_{2\ell}) \in A^{2\ell}\ :\ a_1 - a_2 = a_3 - a_4 = \cdots = a_{2\ell-1}-a_{2\ell} \}}. \]
We note that $E_2^+(A)$ is the standard additive energy $E^+(A)$. 
For applications of higher moment energies, see for example \cite{schoen15,SS13}. 
    \begin{figure}[ht]
            \centering
    \includegraphics[scale=0.6]{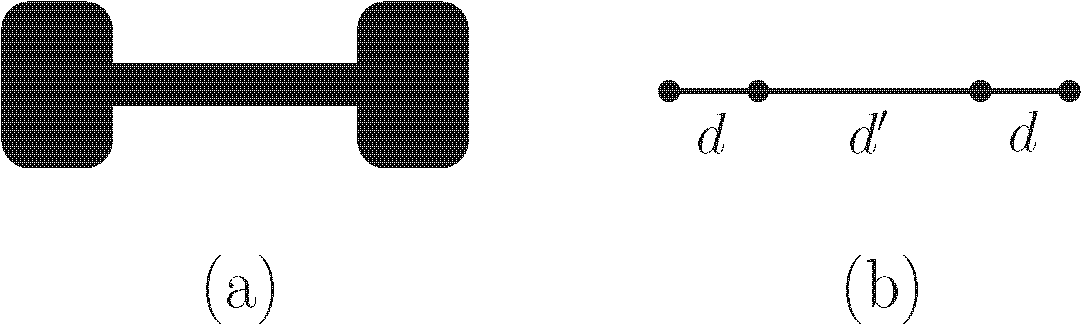} 
            \caption{(a) A dumbbell. (b) An arithmetic dumbbell. }
            \label{fi:dumbbells}
    \end{figure}

We now define a new variant of additive energy, which we call \emph{dumbbell energy}.
We say that a quadruple $(a_1,a_2,a_3,a_4)\in \R$ forms a \emph{dumbbell} if $a_1<a_2<a_3<a_4$ and $a_2-a_1=a_4-a_3$. 
See Figure \ref{fi:dumbbells}.
For distances $d,d'>0$, we define $D(d,d')$ as the congruence class of all dumbbells $(a_1,a_2,a_3,a_4)$ that satisfy $a_2-a_1=a_4-a_3=d$ and $a_3-a_2=d'$.
That is, two dumbbells are congruent if there exists a translation of $\R$ that takes one dumbbell to the other.

We can think of the $\ell$th additive energy as the number of $\ell$-tuples of congruent intervals.
Similarly, we consider an energy variant that counts $\ell$-tuples of congruent dumbbells. 
We define the $\ell$th \emph{dumbbell energy} of $A\subset \R$ as
\begin{align*} 
E^D_{\ell}(A) = \Big|\Big\{(a_1, \ldots, a_{4 \ell}) \in A^{4\ell}\ :\ &\text{ For } 1\le i \le \ell, \text{ the quadruples } (a_{4i-3},a_{4i-2},a_{4i-1},a_{4i}) \\[2mm] 
& \text{ form } \ell \text{ distinct dumbbells from the same congruence class }\Big\} \Big|.
\end{align*}

For $d,d' \in A-A$, we set 
\[ r^D_A(d,d') = |\{(a_1,a_2,a_3,a_4)\in A^4\ :\ a_2-a_1 = a_4-b_3= d \text{ and } a_3-a_2=d' \}|. \]
In other words, $r^D_A(d,d')$ is the number of dumbbells in the congruence class $D(d,d')$ that are spanned by $A$.
We note that a dumbbell $(a_1,a_2,a_3,a_4)$ of $D(d,d')$ is uniquely determined by $a_1$.
Thus, $r^D_A(d,d')\le |A|$ for all $d,d'\in A-A$.
We also note that
\[ E^D_{\ell}(A) = \sum_{d,d'\in A-A}\binom{r^D_A(d,d')}{\ell}. \]

For a difference $d \in A-A$, we set 
\[ r^{-}_A(d) = |\{(a,a')\in A^2\ :\ a-a' = d \}|. \]
In other words, $r^{-}_A(d)$ is the number of representations that the difference $d$ has in $A$.

The following lemma states a connection between the $\ell$th dumbbell energy $E^D_{\ell}(A)$ and the size of the difference set $|A-A|$.

\begin{lemma}\label{le:DumbbellLower}
The following holds for every even constant $\ell\ge 2$. 
Every set $A$ of $n$ reals satisfies that $|A-A| = \Omega\left(n^{4/3} \right)$ or that
\[ E^D_{\ell}(A) = \Omega\left( \frac{n^{4\ell}}{|A-A|^{3\ell - 2}} \right).\]
\end{lemma}

\begin{proof}

As a first step, we will prove that
\[\sum_{d,d' \in A - A} r^D_A(d,d')^2 = \Omega \left( \frac{n^8}{|A-A|^4} \right).\]


Let 
\[ \mathbf{1}_B(x) = \begin{cases} 1 &\text{if } x \in B, \\ 0 &\text{otherwise}, \end{cases}\]
be the indicator function for a set $B \subset \mathbb{R}$. We may write
\[ S := \sum_{d, d' \in A - A} r_A^D(d,d') = \sum_{a \in A} \sum_{d, d'\in A-A} \mathbf{1}_{A \cap (A-d)}(a) \mathbf{1}_{A \cap (A-d)}(a+ d') = \sum_{a\in A} \sum_{d \in A-A} \mathbf{1}_{A \cap (A-d)}(a) r_A^{-}(d). \]
Note that $\sum_{d \in A-A} r_A^{-}(d) = \sum_{a \in A} \sum_{d \in A-A} \mathbf{1}_{A \cap (A-d)}(a) = \binom{|A|}{2} = \binom{n}{2}$.

Consider $\mathcal{D} = \left \{d: |r_A^{-}(d)| \geq S/4 \cdot \binom{n}{2}^{-1} \right \}$. It is not difficult to see that $\mathcal{D}$ captures a large fraction of the above sum:
since \[ \sum_{a \in A} \sum_{d \not \in \mathcal{D}} \mathbf{1}_{A \cap (A-d)}(a) r_A^{-}(d) < \max_{d \not \in \mathcal{D}} r_A^{-}(d) \cdot \binom{n}{2} =  \frac{S}{4}, \]
it follows that
\[ \sum_{a \in A} \sum_{d \in \mathcal{D}} \mathbf{1}_{A \cap (A-d)}(a) r_A^{-}(d) \geq \frac{3S}{4}. \]
We also have the following bound 
\[ |\mathcal{D}| \cdot S/4 \cdot \binom{n}{2}^{-1} \leq \sum_{d \in A-A} r_A^{-}(d)  = \binom{n}{2}.\]
That is, $ |\mathcal{D}| \leq 4/S \cdot \binom{n}{2}^2$.

Let $\mathcal{B} \subset (A-A) \times (A-A) \subset \mathbb{R}^2$ be the set which projects onto a copy of $\mathcal{D}$ on each coordinate axes. By the symmetry in $d, d'$, the above inequality implies that 
\[ \sum_{d,d' \in \mathcal{B}} r_A^{D}(d,d') \geq \frac{3S}{4}. \]
Furthermore, the Loomis-Whitney inequality gives the upper bound $|\mathcal{B}| \leq |\mathcal{D}|^2 \leq 16/S^2 \cdot \binom{n}{2}^4$. Consequently, the Cauchy-Schwarz inequality implies that

\[ \frac{9S^2}{16} \leq \left( \sum_{d, d' \in \mathcal{B}} r_A^{D}(d,d') \right)^2 \leq |\mathcal{B}| \sum_{d,d' \in \mathcal{B}}r_A^D(d,d')^2 \leq 16/S^2 \cdot \binom{n}{2}^4 \sum_{d,d' \in A}r_A^D(d,d')^2.  \]
Equivalently, 
\begin{equation}\label{eq:LW}
    \sum_{d,d' \in A - A} r^D_A(d,d')^2 \geq \frac{9S^4}{256 \cdot \binom{n}{2}^4}.
\end{equation}

By definition, we have that
\begin{align*} 
S = \sum_{d,d'\in A-A} r^{D}_A(d,d') = \sum_{d\in A-A} \binom{r^D_A(d)}{2} &= \sum_{d\in A-A} \frac{r^D_A(d)^2}{2} -\sum_{d\in A-A}\frac{r^D_A(d)}{2} \\[2mm]
&= \sum_{d\in A-A} \frac{r^D_A(d)^2}{2} -\frac{1}{2} \binom{n}{2}.
\end{align*}

Combining the above with the Cauchy-Schwarz inequality leads to
\[ \sum_{d,d'\in A-A} r^{D}_A(d,d') \ge  \frac{(\sum_{d\in A-A}r^D_A(d))^2}{2|A-A|}-\frac{1}{2} \binom{n}{2} = \binom{n}{2}^2/2|A-A| - \frac{1}{2} \binom{n}{2}.\]
If $|A-A|=\Theta(n^2)$ then we are done. We may thus assume that 
\begin{equation} \label{eq:sumDumbbells}
S =\Omega\left(\frac{n^4}{|A-A|}\right).\end{equation}

Combining this Equation~\ref{eq:LW}, we get the desired lower bound 
\[\sum_{d,d' \in A - A} r^D_A(d,d')^2 = \Omega \left( \frac{n^8}{|A-A|^4} \right).\]


Lastly, by H\"older's inequality and since $\ell$ is constant, we obtain that
\begin{align*}
E_\ell^D(A) &= \sum_{d,d'\in A-A} \binom{r_A^D(d,d')}{\ell} \ge  \sum_{d,d'\in A-A} \left(\frac{r_A^D(d,d')}{\ell}\right)^{\ell}
\ge \frac{(\sum_{d,d' \in A-A} r_A^D(d,d')^2)^{\ell/2}}{|A-A|^{2(\ell/2-1)}\ell^{\ell}} \\[2mm]&= \Omega \left( \frac{(n^8/|A-A|^4)^{\ell/2}}{|A-A|^{\ell-2}\ell^{\ell}} \right) = \Omega\left(\frac{n^{4\ell}}{|A-A|^{3\ell-2}}\right).
\end{align*}
\end{proof}

We next introduce a variant of the energy graph that is defined in \cite{FPS20}.
For a finite $A\subset \R$ and $\ell$ even, we define the $\ell$th \emph{dumbell energy graph} $G$ of $A$, as follows. The graph $G$ has a vertex for every $2\ell$-tuple from  $A^{2\ell}$.
There is an edge between the vertices $(a_1,\ldots,a_{2\ell})$ and $(b_1,\ldots,b_{2\ell})$ if
the tuple $(a_1,\ldots,a_{2\ell},b_1,\ldots,b_{2\ell})$ contributes to $E^D_{\ell}$.
Equivalently, if each vertex consists of $\ell/2$ dumbbells and all $\ell$ dumbbells are congruent. 
By definition, the number of edges in $G$ is $E_{\ell}^D(A)$.

We are now ready to prove Theorem \ref{th:Lower43}.
We first recall the statement of this theorem.
\vspace{2mm}

\noindent {\bf Theorem \ref{th:Lower43}.} \emph{Every $k\ge 8$ that is a multiple of 8 satisfies}  
\[ g\left(n, k, \frac{9k^2}{32} - \frac{9k}{8} + 5 \right) = \Omega(n^{4/3}). \]

\begin{proof}
Let $A$ be a set of $n$ reals, such that every subset $A'\subset A$ of size $k$ satisfies that $|A'-A'|\ge \frac{5k^2}{16} - k + 5$.
Let $G$ be the $k/4$ dumbbell energy graph of $A$.
Consider three vertices $v_1,v_2,v_3$ of $G$, such that there is an edge between $v_1$ and $v_2$ and another edge between $v_2$ and $v_3$.
Then each of the three vertices consists of $k/8$ dumbbells, and all $3k/8$ dumbbells are congruent. 
Thus, there is also an edge between $v_1$ and $v_3$.
We conclude that every connected component of $G$ is a clique. 

We may discard connected components with a single vertex without changing the number of edges.
Every remaining connected component corresponds to a distinct dumbbell congruency class.
Since there are $O(|A-A|^2)$ distinct (non-congruent) dumbbells, the number of remaining connected components in $G$ is $O(|A-A|^2)$.

Assume for contradiction that $A$ spans $k/4$ disjoint congruent dumbbells. 
Let $A'\subset A$ be the set of $k$ numbers that form these $k/4$ dumbbells. 
At most four distinct differences are spanned by pairs of numbers from the same dumbbell. 
In Figure \ref{fi:dumbbells}(b), these distances are $d,d',d+d'$, and $2d+d'$.

    \begin{figure}[ht]
            \centering
    \includegraphics[scale=0.6]{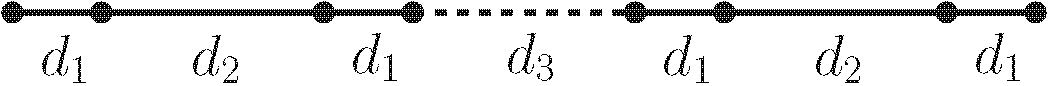} 
            \caption{The nine distances between the two dumbbells are $d_3,d_1+d_3,2d_1+d_3,d_1+d_2+d_3,2d_1+d_2+d_3,3d_1+d_2+d_3,2d_1+2d_2+d_3,3d_1+2d_2+d_3,4d_1+2d_2+d_3$. }
            \label{fi:TwoDumbbells}
    \end{figure}

There are $\binom{k/4}{2}$ pairs of distinct dumbbells as described above. Each such pair spans at most 9 distinct differences with one number from each dumbbell. 
See Figure \ref{fi:TwoDumbbells}.
We conclude that 
\[ |A'-A'| \le 4+\binom{k/4}{2}\cdot 9 = \frac{9k^2}{32}-\frac{9k}{8}+4. \]
This contradicts the assumption, so $A$ does not span $k/4$ disjoint congruent dumbbells. 

Let $S\subset A$ be a dumbbell. 
We wish to derive an upper bound for the number of dumbbells $S'$ that are congruent to $S$ but are not disjoint to it. 
There are at most four choices for the element of $S$ that also appears in $S'$. Then, there are at most three possible positions for this element in $S'$. 
Since $S'$ is congruent to $S$, the other three coordinates of $S'$ are uniquely determined.
Thus, at most 12 dumbbells are congruent to $S$ and not disjoint to it. 

By combining the two preceding paragraphs, we conclude that $A$ cannot span $3k$ congruent dumbbells. 
For a fixed dumbbell congruence class, every vertex in the corresponding connected component of $G$ consists of $k/8$ copies of that dumbbell, possibly with repetitions.
Since there are fewer than $3k$ such dumbbells to choose from, this connected component  consists of fewer than $(3k)^{k/8}$ vertices. 
The number of edges in the connected component is smaller than $(3k)^{k/4}$.
This holds for every connected component of $G$.
Recalling that the number of remaining connected components in $G$ is at most $|A-A|^2$, we conclude that the number of edges in $G$ is $O_k(|A-A|^2)$.

Since the number of edge in $G$ is $E_{k/4}^D(A)$, we have that $E_{k/4}^D(A) = O_k(|A-A|^{2})$.
By Lemma \ref{le:DumbbellLower}, we have that
\[ E^D_{k/4}(A) = \Omega\left( \frac{n^k}{|A-A|^{3k/4 - 2}} \right). \]
Combining the two above bounds for $E^D_{k/4}(A)$ implies that $|A-A|=\Omega(n^{4/3})$.
\end{proof}

\begin{remark}
When performing the above analysis with $E_2^D(A)$ instead of $E^D_{k/4}(A)$, we obtain the weaker bound $|A-A|=\Omega(n^{8/7})$ with the same local condition.
\end{remark}

\begin{remark}\label{rmk:higher-dumbbell}
We say that an 8-tuple $(a_1,a_2,a_3,a_4,a_5,a_6,a_7,a_8)\in \R^8$ forms a \emph{double dumbbell} if $a_1<a_2<a_3<a_4<a_5<a_6<a_7<a_8$ and $a_2-a_1=a_4-a_3 = a_6 - a_5 = a_8 - a_7$ and $a_3 - a_2 = a_7 - a_6$. A natural generalization of the dumbbell energy studied above is the $\ell$th \emph{double dumbbell energy}, defined for $A\subset \R$ as
\begin{align*} 
\mathcal{E}^D_{\ell}(A) = \Big|\Big\{(a_1, \ldots, a_{8 \ell}) \in A^{8\ell}\ :\ &\text{ For } 1\le i \le \ell, \text{ the quadruples } (a_{8i-7}, a_{8i-6}, \ldots, ,a_{8i}) \\[2mm] 
& \text{ form } \ell \text{ distinct double dumbbells from the same congruence class }\Big\} \Big|.
\end{align*}
As before, two double dumbbells are congruent if there exists a translation of $\R$ that takes one double dumbbell to the other.

An identical argument as that for Theorem~\ref{th:Lower43} gives the bound 
\[ g\left(n, k, \frac{31k^2}{128} - \frac{31k}{16} + 9 \right) = \Omega(n^{8/7})\]
for every $k$ a multiple of 16. Here we have slightly weaker local condition than that in Theorem~\ref{th:Lower43} at the expense of a worse exponent in $\Omega\left( n^{8/7} \right)$. We chose to present Theorem~\ref{th:Lower43} because it is less calculation intensive and better illustrates the main ideas of the proof.
\end{remark}

\section{Additional proofs for large values of $\ell$} \label{sec:MoreLargeEll}

In this section, we prove our bounds for large values of $\ell$ that are not related to dumbbell energy. That is, we prove Theorem Theorem \ref{th:CubeLower} and Claim \ref{cl:Quadratic}.

We begin with the proof of Claim 
\ref{cl:Quadratic}, since this simple proof is a nice warm-up. The proof relies on the observation that $a - b = c-d$ implies $a - c = b -d$. 

Recall that for a set $A$ and a difference $d \in A - A$ we set 
\[ r^{-}_A(d) = |\{(a,a')\in A^2\ :\ a-a' = d \}|. \]
\vspace{2mm}

\noindent { \bf Claim \ref{cl:Quadratic}.}
\emph{When $k$ is a multiple of two, we have that} 
\[ g\left(n, k, \frac{3k^2}{8} - \frac{3k}{4} + 2 \right) = \Omega(n^2). \]
\begin{proof}
Let $A$ be a set of $n$ reals, such that every subset $A'\subset A$ of size $k$ satisfies that $|A'-A'|\ge \frac{3k^2}{8} - \frac{3k}{4} + 2$. 
Assume for contradiction that there exists $d\in A-A$ for which $r_A^-(d)\ge k$.
Since every number participates in at most two representations of $d$, there exist distinct  $a_1,\ldots,a_{k/2},b_1,\ldots,b_{k/2}$ such that $a_1-b_1=\cdots = a_{k/2}-b_{k/2}=d$.
We set $S=\{a_1,\ldots,a_{k/2},b_1,\ldots,b_{k/2}\}$.

We may assume, without loss of generality, that $a_1>a_2>\cdots>a_{k/2}$.
This implies that $b_1>b_2>\cdots>b_{k/2}$.
For every $1\le i < j \le k/2$, we have that $a_i-b_i=a_j-b_j$, which in turn implies that $a_i-a_j=b_i-b_j$.
Since there are $\binom{k/2}{2}$ such choices of $i$ and $j$, there are $\binom{k/2}{2}$ equations of the form $a_i-a_j=b_i-b_j$.
We conclude that 
\[ |S-S|\le \binom{k}{2} - \left(\frac{k}{2}-1\right) - \binom{k/2}{2} = \frac{3k^2}{8}-\frac{3k}{4}+1.\]

Since the above contradicts the assumption about $A$, we get that every $d\in A-A$ satisfies $r_A^-(d)< k$.
Since $k$ is constant, we have that
\[ |A-A| > \frac{1}{k} \binom{n}{2} = \Theta(n^2). \]
\end{proof}

We now move to prove Theorem \ref{th:CubeLower}. 
We rely on the following configuration from \cite{FLS19}. Fish, Lund and Sheffer utilize the configuration to obtain an upper bound on $g(n,k,k^{2/3})$. We will instead use this construction to prove a lower bound on $g(n,k,\ell)$ where $\ell = \Theta(k^{2/3})$.   

For an integer $i\ge 1$ and positive reals $a,\delta_1, \ldots, \delta_i$, we define a \emph{projected-$i$-cube} as 
\[ P(a,\delta_1, \ldots, \delta_i) = \{a+x_1\delta_1 + x_2\delta_2 +\cdots + x_i\delta_i\ :\ x_1,\ldots,x_i \in \{0,1\}\}. \]
We note that a projected-$i$-cube is a set of at most $2^i$ real numbers.
A projected-$i$-cube can be thought of as a projection of an $i$-dimensional hypercube onto $\R$, which explains the name of this object.

For $i>1$, every projected-$i$-cube is the union of two projected-$(i-1)$-cubes that are identical up to a translation of distance $p_i$.
The following lemma states another useful property of projected-$i$-cubes.

\begin{lemma}\label{le:cubeDiffs}
If $A$ is a projected-$i$-cube then  
\[ \abs{A-A} \leq \frac{3^{i} - 1}{2}. \]
\end{lemma}

For comparison, a set $A$ of $2^{i}$ random reals is expected to satisfy $\abs{A-A} =\Theta(4^{i})$.

\begin{proof}
The proof is by induction on $i$. 
For the induction basis, we consider the case of $i=1$.
In this case $A$ is a set of two numbers, so $|A-A| = 1 = (3^1-1)/2$.
    
For the induction step, we assume that the claim holds for projected-$i$-cubes and consider a projected-$(i+1)$-cube $A=P(a,\delta_1, \ldots, \delta_{i+1})$.
Then $A$ is the union of the two projected-$i$-cubes $A_1 = P(a,\delta_1, \ldots, \delta_{i})$ and $A_2 = P(a+\delta_{i+1},\delta_1, \ldots, \delta_{i})$.
These two projected-$i$-cubes have the same difference set $D=A_1-A_1=A_2-A_2$.
The induction hypothesis implies that $|D|\le (3^i-1)/2$.

Next, we note that every difference in $(A-A)\setminus\{\delta_{i+1}\}$ is either formed by two numbers from $A_1$ or by one number from $A_1$ and another from $A_2$. 
By the preceding paragraph, there are at most $(3^i-1)/2$ differences of the former type. 
Every difference of the latter type is in $D\pm \delta_{i+1}$.
The number of such differences is at most $2\cdot |D| = 3^i-1$.
In total, 
\[ |A-A|\le (3^i-1)/2 + (3^i-1)+1 = (3^{i+1}-1)/2. \]
This completes the proof of the induction step and thus the proof of the theorem.
\end{proof}

We define the \emph{left endpoint} of a projected-$i$-cube as the smallest number of that cube. The \emph{right endpoint} is the largest number of the cube.
\vspace{2mm}

\noindent {\bf Theorem  \ref{th:CubeLower}.}
\emph{For every $k \geq 8$ that is a power of two and sufficiently large $n$, we have that}
\[ g\left(n,k, \frac{k^{\log_2(3)}+1}{2} \right) = \Omega \left( n^{1+\frac{2}{k-2} }\right).\]

\begin{proof}
We consider a set $A\subset \R$ such that $|A|=n$ and $|A-A|\le n^{1+\frac{2}{k-2}}/9$. 
We show that such a set must contain a projected-$(\log_2(k)-2)$-cube.
Then, Lemma \ref{le:cubeDiffs} implies that $A$ does not satisfy the local property in the statement of the theorem, which completes the proof.

For an integer $i$, we set $\eta_i = 1 - (2^{i+1}-2)/(k-2)$.
We prove by induction on $1\le i \le \log_2 (k) -2$ that $A$ contains at least $n^{\eta_i}$ disjoint translated copies of a projected-$i$-cube with exactly $2^i$ elements.
Combining this with the preceding paragraph completes the proof of the theorem. 

There are $\binom{n}{2}> n^2/3$ pairs of elements from $A$ and $|A-A|\le n^{1+\frac{2}{k-2}}/9$.
By the pigeonhole principle there exists a distance that occurs at least  $3n^{(k-4)/(k-2)}$ times.
We arbitrarily fix one such distance and denote it as $d_1$. 
We set $A_1$ to be the set of elements of $A$ that are at distance $d_1$ from another element of $A$. 
Let $G_1$ be the graph whose vertices are the elements of $A_1$.
Two vertices are connected by an edge if they span the difference $d_1$.
We note that $G_1$ contains at least $3n^{(k-4)/(k-2)}$ edges.
Since every vertex is of degree at most two, we can obtain a matching by repeatedly choosing an arbitrary edge $e$ and discarding at most two other edges that share a vertex with $e$. This implies that $G_1$ contains a matching of size at least $n^{(k-4)/(k-2)}$.
Since a projected-1-cube consists of two points and $\eta_1=(k-4)/(k-2)$, we obtain the case of $i=1$.
This completes the induction basis.

For the induction step, we fix $1\le i < \log_2 (k) -2$ and assume that the claim holds for $i$.
That is, $A$ contains at least $n^{\eta_i}$ disjoint translated copies of a projected-$i$-cube.
We note that, in this range, $\eta_i>0$.
We fix exactly $n^{\eta_i}$ translated copies and set $A_{i+1}$ to be the set of endpoints of these. By definition, we have that $|A_{i+1}|= 2n^{\eta_i}$.
The number of pairs of endpoints that do not belong to the same translated copy is 
\[ \frac{1}{2}\cdot 2 n^{\eta_i}\cdot \left(2 n^{\eta_i}-2\right) = 2n^{2\eta_i}- 2n^{\eta_i}. \]

Each difference is spanned by $A_{i+1}$ fewer than $2n^{\eta_i}$ times. 
Thus, the number of above pairs that span a difference that is also spanned by the projected-$i$-cube is $O(n^{\eta_i})$.
When $n$ is sufficiently large, we may assume that the number of pairs that span a distance that is not in the projected-$i$-cube is at least $n^{2\eta_i}$. Since $A_{i+1}\subset A$ and $|A-A|\le n^{1+\frac{2}{k-2}}/9$, there exists a difference that is spanned by at least 
\[ 9n^{2\eta_i-1-2/(k-2)} = 9n^{\eta_{i+1}} \] 
of these pairs. 
We fix such a difference and denote it as $d_{i+1}$.
In each of the corresponding pairs, the difference $d_{i+1}$ occurs between the two left endpoints, between the two right endpoints, or between one right endpoint and one left endpoint. 
By the pigeonhole principle, at least one third of the pairs are of the same type.
Thus, $A_{i+1}$ contains at least $3n^{\eta_{i+1}}$ translatd copies of the same projected-$(i+1)$-cube. 

The above translated copies of a projected-$(i+1)$-cube might not be disjoint.
In particular, a projected projected-$i$-cube from $A_i$ might participate in two projected-$(i+1)$-cubes.
By repeating the matching argument from the induction basis, we obtain that at least one third of the projected-$(i+1)$-cubes are disjoint. 
That is $A_{i+1}$ contains at least $n^{\eta_{i+1}}$ disjoint translated copies of the same projected-$(i+1)$-cube. 
This concludes the proof of the induction step and thus the proof of the theorem. 
\end{proof}

\section{Small values of $\ell$} \label{sec:SmallEll}

In this section, we prove our bounds for small values of $\ell$:
Theorem \ref{th:SuperLinear} and Theorem \ref{th:ProbConst}. We first briefly go over the tools that we require from additive combinatorics. 

\parag{Tools from additive combinatorics.}
Szemer\'edi's theorem states that every sufficiently large subset of $\{1,2,\ldots,n\}$ contains a long arithmetic progression. We rely on the following variant of this result, due to Gowers \cite{G01}.

\begin{theorem} \label{th:GowersSzem}
For every $k$, there exists $c>0$ that satisfies the following.  Every set of $n/(\log\log n)^c$ numbers from $\{1,2,\ldots,n\}$ contains a $k$-term arithmetic progression. Moreover, $c$ can be taken
to be $2^{-2^{k+9}}$.
\end{theorem}

A \emph{generalized arithmetic progression of dimension} $d$ is defined as
\begin{equation} \label{eq:GAP} \Big\{a+\sum_{j=1}^d k_jb_j\ :\ a,b_1,\ldots,b_d\in \R \text{ and with integer } 0\le k_j \le n_j-1 \text{ for every } 1\le j \le d \Big\}. 
\end{equation}
The size of a generalized arithmetic progression is the number of elements in it.
We note that an arithmetic progression is a generalized arithmetic progression of dimension one. 

One of the most common tools for studying sets with a small different set is \emph{Freiman's theorem.} 
The following is a variant of this theorem over the reals.
For example, see Sanders \cite[Theorem 1.3]{S12}.

\begin{theorem}\label{th:Freiman}
For every sufficiently large finite set $A \subset \R$ and sufficiently large $d > 0$, if $|A - A| \le d|A|$ then $A$ is contained in a generalized arithmetic progression of dimension at most $d\cdot \log^4 d$ and size at most $|A|\cdot e^{d\cdot \log^4 d}$.
\end{theorem}

Similarly to a difference set $A-A$, we define the \emph{sum set} of a set $A\subset \R$ as
\[ A+A = \{a+a'\ :\ a,a\in A\}. \]
The following \emph{Pl\" unnecke-Ruzsa estimate} (for example, see \cite[Corollary 6.29]{TV06}) shows one connection between $A+A$ and $A-A$.

\begin{lemma} \label{le:RuzsaPlun}
If a finite $A\subset \R$ satisfies $|A-A|\le c|A|$ then 
\[ |A+A| \le c^2|A|. \]
\end{lemma}

The following is a special case of a result of Green and Morris \cite{GM16} about sum sets.

\begin{theorem} \label{th:subsets}
For every $c>0$, every sufficiently large $m$ satisfies the following. 
For every integer $n$, the number of subsets $A\subset\{1,2,\ldots,n\}$ that satisfy $|A|=m$ and $|A+A|\le c|A|$ is at most
\[ 2^m \cdot n^{c+1} \cdot \binom{cm/2}{m} . \]
\end{theorem}

\parag{Our proofs.}
We now derive the super-linear threshold for our local properties problem. We first recall the statement of this result.
\vspace{2mm}

\noindent {\bf Theorem \ref{th:SuperLinear}.}
\emph{For every $k$ and sufficiently large $n$, we have that}
\begin{align*} 
g(n,k,k-1) = n-1 \qquad \text{ and } \qquad g(n,k,k) =\omega(n).
\end{align*}

\begin{proof}
To see that $g(n,k,k-1)=n-1$, we set $A=\{1,2,\ldots,n\}$ and note that $|A-A|=n-1$. The local property trivially holds, since every $k$ numbers span at least $k-1$ differences.
It remains to prove that $g(n,k,k) =\omega(n)$.

For a constant $c>0$, we consider a set $A$ such that $|A|=n$ and $|A-A|\le cn$.
By Theorem \ref{th:Freiman}, the set $A$ is contained in a generalized arithmetic progression $G$ of size at most $Cn$ and dimension at most $D$.
The constants $C$ and $D$ depend on $c$ but not on $n$.
We define $G$ as in \eqref{eq:GAP}. Without loss of generality, we may assume that $n_1\ge n^{1/D}$.

We may think of $G$ as the union of disjoint arithmetic progressions with step $b_1$, each of size at least $n_1\ge n^{1/D}$.
Since $|A|/|G|\ge 1/C$, there exists such an arithmetic progression $S$ such that $|A\cap S|/|S|\ge 1/C$.
We translate and scale $\R$ such that $S$ becomes $\{1,2,3,...,|S|\}$.
Then, $A$ contains at least $|S|/C$ elements of $\{1,2,3,...,|S|\}$.
Theorem \ref{th:GowersSzem} implies that $A$ contains a $k$-term arithmetic progression.

We proved that the following holds for every constant $c>0$, when $n$ is sufficiently large: Every set $A$ of $n$ reals that satisfies $|A-A|\le cn$ contains $k$ numbers that span $k-1$ differences.
Thus, a set that satisfies the local restrictions of $g(n,k,\ell)$ also satisfies that $|A-A|>cn$ for every $c$.
\end{proof}

\begin{remark}
Quantitatively, the proof above gives the bound $g(n,k,k) \leq n \cdot \sqrt{ \log(\log \log n)^c}$, where $c = 2^{-2^{k+9}}$.
\end{remark}

For our probabilistic arguments, we recall the following Chernoff bounds (for example, see \cite[Corollary A.1.14]{AS04}).

\begin{theorem} \label{th:Chernoff}
For every $\eps>0$, there exists $c_\eps$ that satisfies the following. 
Let $Y$ be a sum of independent indicator random variables and let $\mu$ be the expected value of $Y$.
Then
\[ \Pr[|Y-\mu|>\eps \mu] < 2e^{-e_\eps\mu}. \]
We may set
\[ c_\eps = \min \{\ln(e^{-\eps}(1+\eps)^{1+\eps}),\eps^2/2\}.  \]
\end{theorem}

We now prove our upper bound for small values of $\ell$. We first recall the statement of this result. \vspace{2mm}

\noindent {\bf Theorem \ref{th:ProbConst}.}
\emph{For every $c\ge 2$ and integer $k> (c^2+1)^2$, we have that} 
\[ g(n, k, ck+1) =O\Big(  n^{1 + \frac{c^2+1}{k}}\Big). \]

\begin{proof}
Let $N$ be a sufficiently large constant multiple of $n$, which is determined below. We define the probability $p = 3N^{-(c^2+1)/k}$ and set $M=\big\{1,2,\ldots, N^{1+(c^2+1)/k} \big\}$. Let $A'$ be a set that is obtained by selecting each element of $M$ independently with probability $p$. The expected size of $A'$ is
\[ p\cdot N^{1+(c^2+1)/k} = 3N. \]

By applying Theorem \ref{th:Chernoff} with $\mu=3N$ and $\eps=1/2$, we obtain that 
\begin{equation} \label{eq:ExpectedSizeA} 
\Pr\left[ \frac{3N}{2} \leq \abs{A} \leq \frac{9N}{2} \right] > 1 - 2e^{-N/4}.
\end{equation}

By Theorem \ref{th:subsets}, the number of subsets $B\subset M$ that satisfy $|B|=k$ and $|B+B|\le c^2k$ is at most 
\[ 2^k \cdot (N^{1+(c^2+1)/k})^{c^2+1} \cdot \binom{c^2k/2}{k} = O_{c,k} (N^{c^2+1+(c^2+1)^2/k}). \]
By the contrapositive of Lemma \ref{le:RuzsaPlun}, the number of subsets $B\subset M$ that satisfy $|B|=k$ and $|B-B|\le ck$ is $O_{c,k} (N^{c^2+1+(c^2+1)^2/k})$.
We refer to such a set as a \emph{set with small doubling}. 

The probability that a fixed set with small doubling is in $A$ is $p^{k}$.
Thus, the expected number of sets with small doubling in $A$ is
\begin{equation*} 
p^{k} \cdot O_{c,k} (N^{c^2+1+(c^2+1)^2/k}) =  O_{c,k}(N^{(c^2+1)^2/k}). 
\end{equation*}

By the assumption $k>(c^2+1)^2$, the expected number of sets with small doubling in $A$ is $o_{c,k}(N)$.
Combining this with \eqref{eq:ExpectedSizeA} implies that, with positive probability, we have that $\frac{3N}{2} \leq \abs{A} \leq \frac{9N}{2}$ and that the number of sets with small doubling is $o_{c,k}(N)$.
Thus, there exists $A$ that satisfies both of these properties.
We fix such a set $A$ and arbitrarily remove one number from each of the remaining sets with small doubling. 
This leads to $|A|=\Theta(N)$ and no sets with small doubling in $A$.
In other words, every $k$ elements from $A$ span at least $ck+1$ distinct differences. 

Since $A\subset M$, we have that $|A-A|< N^{1+(c^2+1)/k}$.
By taking $N$ to be a sufficiently large constant mulitple of $n$, we obtain that $|A|\ge n$. We arbitrarily remove elements from $A$ until $|A|=n$.
This leads to the assertion of the theorem. 
\end{proof}

\section{Future work}\label{sec:Future}
We are still far from understanding the behavior of $g(n,k, \ell)$ an have more questions than answers. One main open problem is to identify the quadratic threshold of $g(n,k,\ell)$. Since this turned out to be challenging, we suggest the following problem as a step towards the quadratic threshold.

\begin{question}
For what values of $\ell$ can we prove a bound of the form $g(n,k, \ell) = \Omega(n^c)$ for some $4/3 < c < 2$?
\end{question}

New ideas would be needed for this problem: as mentioned in Remark~\ref{rmk:higher-dumbbell}, though the idea of dumbbell energies can be generalized to ``higher'' dumbbell energies, the bounds obtained that way would be of the form $g(n,k,\ell) = \Omega(n^{c})$ for some $c < 4/3$. 

In terms of quadratic lower bounds, we pose the following question.

\begin{question}
Is it true that when $k$ is a multiple of four, we have that 
\[ g\left(n, k, \frac{k^2}{4} \right) = \Omega(n^2)?\]
\end{question}

There is also a natural generalization of the construction of the projected-$(i+1)$-cube, where instead of projecting an $i$-dimensional Boolean hypercube onto $\mathbb{R}$, we could instead project an $i$-dimensional $b$-ary hypercube onto $\mathbb{R}$. Specifically, let $A_1 = \{1,2, \ldots, b\}$ and $D = \{0, 1, \ldots, b-1 \}$. 
For every $i>1$, we let $s_i = 4b \cdot \max \{ A_{i-1} - A_{i-1}\}$ and $T_i = s_i\cdot D$.
We then set 
\[ A_i = T_i + A_{i-1}. \] 

We note that $1\in A_i$ for every $i\ge 1$.
This implies that $s_i\ge 4b (\max\{A_i\}-1)$.
Since the nonzero elements of $T_i$ are significantly larger than the elements of $A_i$, we get that $|A_i|=b\cdot |A_{i-1}|$.
Since $|A_1|=b$, we conclude that $|A_i|=b^i$ for every $i\ge1$.

By definition, every element of $A_i - A_i$ can be represented as $x + y$ where $x \in T_i - T_i$ and $y \in A_{i-1} - A_{i-1}$. 
By our choice of $s_i$, we have that $x, x' \in T_i - T_i$ and $y, y' \in A_{i-1} - A_{i-1}$ satisfy that $x + y = x' + y'$ if and only if $(x,y)=(x',y')$. Indeed, if $x\neq x'$ then $\abs{x - x'} \geq 4b \max\{ |y|,|y'| \} \geq 2b \abs{y - y'}$. In particular, it can be easily checked that as a consequence we have $\abs{A_i-A_i} = \frac{(2b-1)^{i}-1}{2}$. We conjecture that the natural generalization of \cite{FLS19} should hold in this case.

\begin{conjecture}
When $k$ is a power of $b$, $g\left(n, k, \frac{k^{\log_b(2b-1)} +1}{2} \right) = O\left(k^{\log_b(2b-1)} \right)$.
\end{conjecture}

In the same flavor as Theorem~\ref{th:CubeLower}, we may ask the following question.

\begin{question}
Can we prove a lower bound for $g\left(n, k, \frac{k^{\log_b(2b-1)} +1}{2} \right)$ when $k$ is a power of $b$?
\end{question}

\bibliographystyle{plain}
\bibliography{main.bib}

\end{document}